\documentclass[a4paper,10pt]{amsart}

\usepackage[utf8]{inputenc}     

\usepackage{amsfonts}
\usepackage{amsthm}
\usepackage{amssymb}
\usepackage{amsmath}
\usepackage{enumerate}
\usepackage{mathtools}
\usepackage{tikz}
\usepackage{url}
\usepackage{nameref} 
\usepackage{hyperref} 
\usepackage{color}

\newtheorem{theorem}{Theorem}

\newtheorem{lemma}[theorem]{Lemma}
\newtheorem{proposition}[theorem]{Proposition}

\newcommand{\FT}[0]{\mathbb{F}_2}

\begin{document}
\title[Normal numbers and nested perfect necklaces]{Normal numbers and 
nested perfect necklaces}

\author{Ver\'onica Becher \and Olivier Carton}
\date{May 7, 2018}
\maketitle

\begin{abstract}
  M. B. Levin used Sobol-Faure low discrepancy sequences with Pascal
  matrices modulo~$2$ to construct, for each integer~$b$, a real number $x$
  such that the first~$N$ terms of the sequence $(b^n x \mod 1)_{n\geq 1}$
  have discrepancy $O((\log N)^2/N)$.  This is the lowest discrepancy known
  for this kind of sequences.  In this note we characterize Levin's
  construction in terms of nested perfect necklaces, which are a variant of
  the classical de Bruijn necklaces.  Moreover, we show that every real
  number $x$ whose base~$b$ expansion is the concatenation of nested
  perfect necklaces of exponentially increasing order satisfies that the
  first $N$ terms of \mbox{$(b^n x \mod 1)_{n\geq 1}$} have discrepancy
  $O((\log N)^2/N)$.  For base~$2$ and the order being a power of~$2$, we
  give the exact number of nested perfect necklaces and an explicit method
  based on matrices to construct each of them.
\end{abstract}
\bigskip



\section{Introduction and statement of results}

For a sequence $(x_n)_{n\geq 0}$ of real numbers in the unit interval the
discrepancy of the first~$N$ elements~is
\begin{displaymath}
  D_N((x_n)_{n\geq 0}) =
  \sup_{0\leq \alpha < \beta \leq 1}\left| \frac{1}{N} \#\Big
    \{ n : 0 \leq n < N  \text{ and } \alpha \leq x_n < \beta \Big\} -
    (\beta-\alpha)
  \right|.
\end{displaymath}
In~\cite{Schmidt1972} Schmidt showed that there is a constant $C$ such that
for \emph{every} sequence~$(x_n)_{n\geq 0}$ of real numbers there are
infinitely many $N$s such that
\begin{displaymath}
  D_N((x_n)_{n\geq 0}) >C \frac{\log N}{N}.
\end{displaymath}
This is an optimal order of discrepancy since this lower bound is achieved
by van der Corput sequences,
see~\cite{KuipersNiederreiter2006,DrmotaTichy1997,Bugeaud2012}.

The property of Borel normality for real numbers can be defined in terms of
uniform distribution.  A sequence $(x_n)_{n\geq 0}$ of real numbers in the
unit interval is uniformly distributed exactly when
$\lim_{N\to\infty} D_N((x_n)_{n\geq 0})=0$.  We write $\{x\} $ to denote
$x-\lfloor x\rfloor$, the fractional part of $x$.  For an integer $b$
greater than $1$, a real number~$x$ is normal to base~$b$ if the sequence
$(\{b^n x\})_{n\geq 0}$, is uniformly distributed.  It is still unknown
whether the optimal order of discrepancy can also be achieved by a sequence
of the form~$(\{b^n x\})_{n\geq 0}$ for some real number~$x$
\cite{Korobov1955,DrmotaTichy1997,Bugeaud2012}.  

The lowest discrepancy
known for sequences of this form is $O((\log N)^2/N)$ and it holds for a
real number~$x$ constructed by Levin in~\cite{Levin1999}.  Given an
arbitrary integer base~$b$, Levin's construction uses Sobol-Faure sequences
with the Pascal triangle matrix modulo~$2$,
see~\cite{Levin1995,Levin1999,Faure1982}.

Our first result in this note is a characterization of Levin's construction
in terms of combinatorics of words, showing that it is a concatenation of
what we call \emph{nested perfect necklaces} of increasing order.  Perfect
necklaces were introduced in~\cite{AlBeFeYu2016} as a variant of the
classical de Bruijn necklaces~\cite{deBruijn1946}.  Fix an alphabet $A$. A
word is a finite sequence of symbols and a necklace, or circular word, is
the equivalence class of a word under rotations.  For positive integers $k$
and~$m$, we call a necklace $(k, m)$-perfect if each word of length~$k$
occurs exactly $m$ times at positions which are different modulo $m$ for
any convention on the starting point.  The length of a $(k,m)$-perfect
necklace is $m|A|^k$ where $|A|$ denotes the cardinality of the
alphabet~$A$.  In this note we always consider the modulo~$m$ being a power
of~$2$.

Notice that for $m = 1$, the $(k,m)$-perfect necklaces are exactly the de
Bruijn necklaces of order~$k$.  For the binary alphabet the word $0011$ is
a $(1,2)$-perfect necklace.  Both words $00110110$ and $00011011$ are
$(2,2)$-perfect necklaces.  The segments in Champernowne sequence which are
the concatenation in lexicographic order of all words of length~$k$ is a
$(k,k)$-necklace. For instance the following word is a $(3,3)$-perfect
necklace  (the spacing is just for the readers convenience),
\begin{displaymath}
  000 \ 001\ 010 \ 011 \ 100 \ 101 \ 110 \ 111
\end{displaymath}
More generally, every arithmetic sequence with difference coprime 
with the alphabet size yields a perfect necklace.

A word~$w$ is a \emph{$(k,m)$-nested perfect necklace} if for each integer
$\ell =1,2, \ldots, k$, each block of~$w$ of length $m|A|^\ell$ which starts
at a position congruent to~$1$ modulo~$m|A|^\ell$ is a $(\ell,m)$-perfect
necklace.  An alternative recursive definition of nested perfect necklaces
is as follows.  A word~$w$ is a $(k,m)$-nested perfect necklace if, first,
it is a $(k,m)$-perfect necklace; and, second, either $k = 1$ or whenever
$w$ is factorized $w = w_1 \cdots w_{|A|}$ with each word~$w_i$ of length
$m|A|^{k-1}$, then each word~$w_i$ is a $(k-1,m)$-nested perfect necklace.

Notice that each $(k,m)$-nested perfect necklace is not an equivalence
class closed under rotations, but it is a single word, with a unique
initial position.  The word $00110110$ is a $(2,2)$-nested perfect necklace
because it is a $(2,2)$-perfect necklace and both words $0011$ and $0110$
are $(1,2)$-perfect necklaces.  The four words
\begin{gather*}
  0000111101011010 \\
  0011110001101001 \\
  0001111001001011 \\
  0010110101111000
\end{gather*}
are $(2,4)$-nested perfect necklaces. Both the concatenation of the first
two and the concatenation of the last two are $(3,4)$-nested perfect
necklaces.  The concatenation of all of them is a $(4,4)$-nested perfect
necklace.  The concatenation of all words of the same length in
lexicographic order yields a perfect necklace that is not a nested perfect
necklace.

The statement of our first result is as follows.
\begin{theorem}\label{thm:1}
  For each base $b$ the number $x$ defined by Levin
  in~\cite[Theorem~2]{Levin1999} using the Pascal triangle matrix modulo~$2$ is
  obtained as the concatenation of the $(m,m)$-nested perfect necklaces
  for $m=2^d$ with $d=0,1,2,\ldots$.  Conversely, for every
  number~$x$ whose base~$b$ expansion is the concatenation of
  $(m,m)$-nested perfect necklaces for $m = 2^d$ with $d = 0,1,2\ldots$,
  the discrepancy $D_N((\{b^n x\})_{n\geq 1})$ is $O((\log N)^2/N)$.
\end{theorem}

To state the second result, we consider the field with two elements $\FT$
and we introduce a family of $2^{m-1}$ matrices of dimension $m \times m$
over $\FT$ obtained by rotating the columns of the Pascal triangle matrix
modulo~$2$.  We identify words of two symbols with vectors,
for each matrix~$M$ of this family, we construct a
nested perfect necklace  by concatenating the words of
the form $Mw \oplus z$
 where $z$ is a fixed word over~$\FT$ of length~$m$ and $w$ ranges over all 
words over~$\FT$ of length~$m$ in lexicographic order. 
 Such a necklace is called an \emph{affine necklace}.
Our second result is as follows.

\begin{theorem}\label{thm:2}
  For each $m=2^d$ with $d=0,1,2,\ldots$ there are $2^{2m-1}$ binary
  $(m,m)$-nested perfect necklaces and they are exactly the affine
  necklaces.
\end{theorem}

The rest of this note is devoted to the proofs.  For the proof of
Theorem~\ref{thm:1}, first notice that Levin's construction in his
Theorem~2 in \cite{Levin1999} is the concatenation of blocks obtained using
Pascal triangle matrix modulo~$2$ for increasing $m = 2^d$ with
$d = 0,1,2,\ldots$.  Hence, each block is an $(m,m)$-affine necklace and,
by Proposition~\ref{pro:aff2nes}, each block is an $(m,m)$-nested perfect
necklace.  Conversely, assume that the expansion in base~$b$ of a given
real~$x$ can be split in consecutive blocks such that each block is an
$(m,m)$-nested perfect necklace for $m = 2^d$ with $d = 0,1,2,\ldots$.
Then, Levin's chain of estimates in \cite{Levin1999} yield the wanted
discrepancy: his Lemma~5, Corollaries 1 and~2 and the end of the proof of
his Theorem~2.

The proof of Theorem~\ref{thm:2} follows from Propositions~\ref{pro:aff2nes}, 
\ref{prop:nbr-nested} and~\ref{pro:nes2aff}.

\section{Affine necklaces}

In this note we consider transformations on words obtained as linear maps
over the field~$\FT$ with two elements.  We identify the words of
length~$n$ over~$\FT$ with the column vectors of dimension $n \times 1$
over~$\FT$.  More precisely, we always identify the word $a_1 \cdots a_n$
where $a_i \in \FT$ with the column vector
$(a_1,\ldots, a_n)^t \in (\FT)_{n\times 1}$ where ${}^t$ denotes transpose
of vectors and matrices. Suppose $w_1,\ldots,w_k$ is a sequence of words,
each of them of length~$n$ and $M$ is a $n \times n$-matrix over~$\FT$, we
may consider the concatenation $(Mw_1)(Mw_2) \cdots (Mw_k)$.  In this
writing, the matrix~$M$ is multiplied with each word~$w_i$ considered as a
column vector, and the resulting column vector is viewed again as a word of
length~$n$.
Similarly, the component-wise sum of vectors in~$\FT$
is used directly on words of the same length.  It is denoted by the
symbol~$\oplus$.

We assume that the alphabet is $\FT = \{0, 1\}$ and that the modulo~$m$ is
always a power of~$2$, namely $m = 2^d$ for some non-negative integer~$d$.
We now define a family of matrices that we will use to construct explicitly
some nested perfect necklaces.  We start by defining by induction on~$d$ an
$m \times m$-matrix~$M_d$ for each $d \ge 0$ by
\begin{displaymath}
  M_0 = (1)
  \quad\text{and}\quad
  M_{d+1} = \begin{pmatrix} M_d & M_d \\ 0  & M_d \end{pmatrix}.
\end{displaymath}

The matrices $M_1$ and $M_2$ are then 
\begin{displaymath}
  M_1 = \begin{pmatrix}
    1 & 1 \\
    0 & 1
  \end{pmatrix} 
  \quad\text{and}\quad
  M_2 = \begin{pmatrix}
    1 & 1 & 1 & 1 \\
    0 & 1 & 0 & 1 \\
    0 & 0 & 1 & 1 \\
    0 & 0 & 0 & 1
  \end{pmatrix}.
\end{displaymath}
\medskip

The matrix~$M_d$ is a variant of the Pascal triangle modulo~$2$
in rectangular form,  we prove it in Lemma~\ref{lem:pascal} below.  
This matrix is almost the one used by Levin in~\cite{Levin1999} because we have
reversed the order of the columns.  This definition of the matrix $M_d$
 allows us to identify words with column vectors, which is not the case in~\cite{Levin1999}.

We now introduce a family of  matrices obtained by applying some rotations to
columns of the matrix~$M_d$. 
Let $\sigma$ be the function which maps each word
$a_1 \cdots a_n$ to $a_na_1a_2 \cdots a_{n-1}$ obtained by moving the last symbol
to the front.  Since words over~$\FT$ are identified with column vectors, the
function~$\sigma$ can also be applied to a column vector.

Let $n_1,\ldots,n_m$ be a sequence of integers such that $n_m = 0$ and
$n_{i+1} \le n_i \le n_{i+1}+1$ for each integer $1 \le i <m$.  
Let $C_1,\ldots,C_m$ be the columns of $M_d$, that is, 
$M_d= (C_1,\ldots,C_m)$.
Define
\[
  M_d^{n_1,\ldots,n_m}= \bigl(\sigma^{n_1}(C_1),\ldots,\sigma^{n_m}(C_m)\bigr).
\]  
The following are  the eight possible matrices $M_d^{n_1,\ldots,n_m}$
 for $d = 2$ and $m = 2^2$.

\begin{displaymath}
  \begin{array}{cccc} 
    M_4^{0,0,0,0} & M_4^{1,0,0,0} & M_4^{1,1,0,0} & M_4^{2,1,0,0} \\
    \begin{pmatrix}
      1 & 1 & 1 & 1 \\
      0 & 1 & 0 & 1 \\
      0 & 0 & 1 & 1 \\
      0 & 0 & 0 & 1
    \end{pmatrix}
    & 
    \begin{pmatrix}
      0 & 1 & 1 & 1 \\
      1 & 1 & 0 & 1 \\
      0 & 0 & 1 & 1 \\
      0 & 0 & 0 & 1
    \end{pmatrix} 
    & 
    \begin{pmatrix}
      0 & 0 & 1 & 1 \\
      1 & 1 & 0 & 1 \\
      0 & 1 & 1 & 1 \\
      0 & 0 & 0 & 1
    \end{pmatrix}
    & 
    \begin{pmatrix}
      0 & 0 & 1 & 1 \\
      0 & 1 & 0 & 1 \\
      1 & 1 & 1 & 1 \\
      0 & 0 & 0 & 1
    \end{pmatrix} \\[10mm]
    M_4^{1,1,1,0} & M_4^{2,1,1,0} & M_4^{2,2,1,0} & M_4^{3,2,1,0} \\
    \begin{pmatrix}
      0 & 0 & 0 & 1 \\
      1 & 1 & 1 & 1 \\
      0 & 1 & 0 & 1 \\
      0 & 0 & 1 & 1
    \end{pmatrix}
    & 
    \begin{pmatrix}
      0 & 0 & 0 & 1 \\
      0 & 1 & 1 & 1 \\
      1 & 1 & 0 & 1 \\
      0 & 0 & 1 & 1
    \end{pmatrix} 
    & 
    \begin{pmatrix}
      0 & 0 & 0 & 1 \\
      0 & 0 & 1 & 1 \\
      1 & 1 & 0 & 1 \\
      0 & 1 & 1 & 1
    \end{pmatrix}
    & 
    \begin{pmatrix}
      0 & 0 & 0 & 1 \\
      0 & 0 & 1 & 1 \\
      0 & 1 & 0 & 1 \\
      1 & 1 & 1 & 1
    \end{pmatrix}
  \end{array}
\end{displaymath}
\bigskip

Let $m = 2^d$ for some $d \ge 0$ and let $k$ be some integer such that
$1 \le k \le m$.  Let $w_1,\ldots,w_{2^m}$ be the enumeration in
lexicographic order of all words on length~$m$ over~$\FT$. Let $z$ be a
word over~$\FT$ of length~$m$ and let $w'_i = w_i \oplus z$ for
$1 \le i \le 2^m$.  Let $M$ be a matrix a matrix $M_d^{n_1,\ldots,n_m}$ as
above.  Then, the concatenation 
\begin{displaymath}
  (Mw'_1)(Mw'_2) \cdots (Mw'_{2^k})
\end{displaymath}
is called an \emph{$(k,m)$-affine} necklace.  In the sequel we refer to
this necklace as the affine necklace obtained from the
matrix~$M_d^{n_1,\ldots,n_m}$ and the vector~$z$.  Note that setting
$z' = Mz$ gives $Mw'_i = Mw_i \oplus z'$ which justifies the terminology.
In Lemma~\ref{lem:invert-right-sm} we will prove that each matrix
$M_d^{n_1,\ldots,n_m}$ is invertible and therefore each vector~$z'$ is
equal to~$Mz$ for some vector~$z$.

Each matrix~$M_d$ is upper triangular, that is $(M_d)_{i,j} = 0$ for
$1 \le i < j \le 2^d$. The following lemma states that the upper part of
the matrix~$M_d$ is the beginning of the Pascal triangle modulo~$2$ also
known as the Sierpi\'nski triangle.

\begin{lemma} \label{lem:pascal}
  For any integers $d,i,j$ such that $d \ge 0$ and $1 \le i,j < 2^d$,
  $(M_d)_{i,j} = (M_d)_{i+1,j} \oplus (M_d)_{i,j+1}$.
\end{lemma}
\begin{proof}
  The proof is carried out by induction on~$d$.  If $d = 0$, the result
  trivially holds.  Suppose that the result holds for $M_d$ and let $i,j$
  be integers such that $1 \le i,j < 2^{d+1}$.  If $i$ and $j$ are
  different from $2^d$, the result follows directly from the induction
  hypothesis and the definition of~$M_{d+1}$.  If either $i = 2^d$ or
  $j = 2^d$, the result follows from the fact that $(M_d)_{i,j}$ is equal
  to $1$ if either $i = 2^d$ or $j = 2^d$ and it is equal to~$0$ if $i = 0$
  or $j = 0$ (and $i$ and $j$ different from $2^d$).  This latter fact is
  easily proved by induction on~$d$.
\end{proof}

\section{Affine necklaces are nested perfect necklaces}

\begin{figure}[htbp]
  \begin{center}
    \begin{tikzpicture}
      \node at (-0.65,0) {$M = $};
      \node at (0,0) {$\left(\vphantom{\rule{0cm}{1cm}} \right.$};
      \node at (2,0) {$\left.\vphantom{\rule{0cm}{1cm}} \right)$};
      \node at (1.55,0.1) {$P$};
      \draw (1.15,-0.3) rectangle (1.95,0.5);
      \draw [<->] (1.15,-0.4) -- (1.95,-0.4);
      \node at (1.55,-0.6) {$k$};
      \draw [<->] (0,-0.4) -- (1.15,-0.4);
      \node at (0.6,-0.6) {$m-k$};
      \draw [<->] (1.05,0.5) -- (1.05,-0.3);
      \node at (0.9,0.1) {$k$};
      \draw [<->] (1.05,0.8) -- (1.05,0.5);
      \node at (0.9,0.65) {$\ell$};
      \draw [<->] (2.15,0.9) -- (2.15,-0.9);
      \node at (2.45,0) {$m$};
      \draw [<->] (0,-0.95) -- (2,-0.95);
      \node at (1,-1.2) {$m$};
    \end{tikzpicture}
  \end{center}
  \caption{Position of the sub-matrix $P$ in $M$
           in Lemma~\ref{lem:invert-right-sm}.}
\label{fig:wallP}
\end{figure}

\begin{lemma} \label{lem:invert-right-sm}
  Let $M$ be 
  a matrix $M_d^{n_1,\ldots,n_m}$.  Let $\ell$ and~$k$
  be two integers such that $0 \le \ell < \ell + k \le 2^d$.  Any
  sub-matrix obtained by selecting the $k$ rows
  $\ell+1,\ell+2,\ldots,\ell+k$ and the last $k$ columns
  $2^d-k+1,\ldots,2^d$ of~$M$ is invertible.
\end{lemma}
Note that for $k = 2^d$ and $\ell = 0$, the sub-matrix 
in the statement of the  lemma ,
is the whole matrix $M_d^{n_1,\ldots,n_m}$,
 which is invertible.
\begin{proof}
  Let $m = 2^d$ be the number of rows and columns of~$M$.
  By Lemma~\ref{lem:pascal}, each entry~$M_{i,j}$ for $1 \le i,j < m$ of
  the matrix~$M$ satisfies either $M_{i,j} = M_{i+1,j} \oplus M_{i,j+1}$ if
  $n_j = n_{j+1}$ (the column~$C_j$ has been rotated as much as the
  column~$C_{j+1}$) or $M_{i,j} = M_{i+1,j} \oplus M_{i+1,j+1}$ if
  $n_j = n_{j+1} + 1$ (the column~$C_j$ has been rotated once more than the
  column~$C_{j+1}$).

  Let $P$ be the sub-matrix in the statement of the lemma,
  a picture appears as Figure \ref{fig:wallP}. 
  To prove that~$P$ is invertible we apply
  transformations to make it triangular.  Note that all entries of the last
  column are~$1$.  The first transformation applied to~$P$ is as follows.
  The row $L_1$ is left unchanged and the row~$L_j$ for
  $2 \le j \le k$ is replaced by $L_i \oplus L_{i-1}$.  All entries of the
  last column but its top most one become zero.  Furthermore, each entry is
  $P_{i,j}$ is either replaced by either $P_{i,j+1}$ or~$P_{i+1,j+1}$
  depending on the value $n_j - n_{j+1}$.  Note also that the new values of
  the entries still satisfy either $P_{i,j} = P_{i+1,j} \oplus P_{i,j+1}$
  or $P_{i,j} = P_{i+1,j} \oplus P_{i+1,j+1}$ depending on the value
  $n_j - n_{j+1}$. 
 The second transformation applied to~$P$ is as follows.
  The rows $L_1$ and~$L_2$ are left unchanged and each
  row~$L_i$ for $3 \le i \le k$ is replaced by $L_i \oplus L_{i-1}$.  All
  entries of the second to last column but its two topmost ones are now
  zero.  At step~$n$ for $1 \le n < k$, rows $L_1,\ldots,L_n$ are
  left unchanged and each row~$L_i$ for $n+1 \le i \le k$ is replaced
  by $L_i \oplus L_{i-1}$.  After applying all these transformations for
  $1 \le n < k$, each entry~$P_{i,j}$ for $i+j = k+1$ satisfies
  $P_{i,j} = 1$ and each entry~$P_{i,j}$ for $i+j > k+1$ satisfies
  $P_{i,j} = 0$.  It follows that the determinant of~$P$ is~$1$ and that
  the matrix~$P$ is invertible.
\end{proof}

We now introduce the notions of {\em upper} and {\em lower border} of a
matrix~$M_d^{n_1,\ldots,n_m}$.  Let $m = 2^d$ for some $d \ge 0$ and let
$M$ be one matrix~$M_d^{n_1,\ldots,n_m}$.  An entry~$M_{i,j}$ for
$1 \le i,j \le m$ is said to be in the \emph{upper border} (respectively
\emph{lower border}) of~$M$ 
if $M_{i,j} = 1$ and $M_{k,j} = 0$ for all $k=1, .. i-1$
 (respectively below). 
For instance, the upper border of the matrix~$M_d$ 
is the first row and its lower border is the main diagonal. 
The following pictures in boldface the upper and lower borders of the matrix 
$M_3^{3,3,2,1,1,1,0,0}$:

\begin{displaymath}
  M_3^{3,3,2,1,1,1,0,0} = 
  \begin{pmatrix}
     0 & 0 & 0 & 0 & 0 & 0 & \mathbf{1} & \mathbf{1} \\
     0 & 0 & 0 & \mathbf{1} & \mathbf{1} & \mathbf{1} & 0 & 1 \\
     0 & 0 & \mathbf{1} & 1 & 0 & 1 & 1 & 1 \\
     \mathbf{1} & \mathbf{1} & 0 & 1 & 0 & 0 & 0 & 1 \\
     0 & \mathbf{1} & \mathbf{1} & \mathbf{1} & 0 & 0 & 1 & 1 \\
     0 & 0 & 0 & 0 & \mathbf{1} & 1 & 0 & 1 \\
     0 & 0 & 0 & 0 & 0 & \mathbf{1} & \mathbf{1} & 1 \\
     0 & 0 & 0 & 0 & 0 & 0 & 0 & \mathbf{1} 
  \end{pmatrix}
\end{displaymath}
\bigskip

 We gather now some easy facts about the upper and lower borders of a
matrix~$M_d^{n_1,\ldots,n_m}$. 
Both borders start in the unique entry~$1$ of the first column. 
The upper border ends in the top most entry of the last column
and the lower border ends in the bottom most entry of the last column.  
The upper border only uses either East or
North-East steps and the lower border only uses either East or South-East
steps.  The upper border uses a East step from column~$C_j$ to
column~$C_{j+1}$ if $n_j-n_{j+1}= 0$ and uses a North-East step if
$n_j-n_{j+1}= 1$.  Furthermore, whenever the upper border uses an East
(respectively North-East) step to go from one columns to its right
neighbour, the lower border uses a South-East (respectively East) step.
This is due to the fact that the distance from the upper border to the
lower border in the $i$-th column is $i-1$.

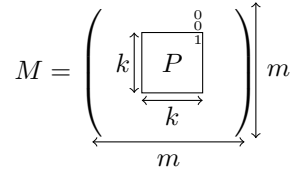
\begin{figure}[htbp]
  \begin{center}
    \begin{tikzpicture}
      \node at (-0.65,0) {$M = $};
      \node at (0,0) {$\left(\vphantom{\rule{0cm}{1cm}} \right.$};
      \node at (2,0) {$\left.\vphantom{\rule{0cm}{1cm}} \right)$};
      \node at (1.05,0.1) {$P$};
      \draw (0.65,-0.3) rectangle (1.45,0.5);
      \draw [<->] (0.65,-0.4) -- (1.45,-0.4);
      \node at (1.05,-0.6) {$k$};
      \draw [<->] (0.55,0.5) -- (0.55,-0.3);
      \node at (0.4,0.1) {$k$};
      \node at (1.39,0.4) {$\scriptscriptstyle 1$};
      \node at (1.39,0.59) {$\scriptscriptstyle 0$};
      \node at (1.39,0.74) {$\scriptscriptstyle 0$};
      \draw [<->] (2.15,0.9) -- (2.15,-0.9);
      \node at (2.45,0) {$m$};
      \draw [<->] (0,-0.95) -- (2,-0.95);
      \node at (1,-1.2) {$m$};
    \end{tikzpicture}
  \end{center}
  \caption{Position of the sub-matrix $P$ in $M$
           in Lemma~\ref{lem:invert-top-sm}}
\label{fig:floatP}
\end{figure}

Due to the symmetry in the matrix $M_d$
Lemma~\ref{lem:invert-right-sm} 
 applies also  to the  sub-matrices of $M_d$
obtained by selecting the first row.  
Since this symmetry is lost for matrices $M_d^{n_1,\ldots,n_m}$.
we need the following lemma which
accounts for the rotations made to the columns in~$M_d$ to obtain
$M_d^{n_1,\ldots,n_m}$.

\begin{lemma} \label{lem:invert-top-sm}
  Let $M$ be one of the matrix $M_d^{n_1,\ldots,n_m}$.  Let $k$ be an
  integer such that $1 \le k \le 2^d$.  The $k\times k$-sub-matrix obtained
  by selecting $k$ consecutive rows and $k$ consecutive columns in such
  a way that such that its top right entry lies in the upper border of~$M$ is
  invertible.
\end{lemma}
\begin{proof}
  The proof is similar to that of Lemma~\ref{lem:invert-right-sm}.  
   Let $P$ be the sub-matrix in the statement of the lemma, a picture appears in Figure \ref{fig:floatP}.  
   We apply transformations to the sub-matrix~$P$ to put it in a nice form  such that 
   the determinant is easy to compute.  Just to fix notation, we
  suppose that the sub-matrix~$P$ is obtained by selecting rows
  $L_{r+1},\ldots,L_{r+k}$ and columns $C_{s+1},\ldots,C_{s+k}$.  The
  hypothesis is that the entry $M_{r+k,s+k}$ is in the upper border
  of~$M$.  Note that the upper borders of $M$ and~$P$ coincide inside~$P$.
  We denote by $j_1,\ldots, j_t$ the indices of the columns of~$P$ in
  $1,\ldots,k$ which are reached by a North-East step of the upper
  border.  This means that $j_1,\ldots,j_t$ is the sequences of indices~$j$
  such that $n_{s+j-1}-n_{s+j} = 1$.  By convention, we set $j_0 = 1$, that
  is, the index of the first column of~$P$.

  The first transformation applied to the matrix~$P$ is the following.  The
  columns $C_1,\ldots,C_{j_t-1}$ and~$C_k$ are left unchanged and each
  column~$C_j$ for $j_t \le j \le k-1$ is replaced by $C_j \oplus C_{j+1}$.
  All entries of the first row but its right most one become~zero.

  Furthermore, each entry $P_{i,j}$ for $j_t \le j \le k-1$ is replaced by
  $P_{i+1,j}$.  The second transformation applied to the matrix~$P$ is the
  following.  The columns $C_1,\ldots,C_{j_{t-1}-1}$ and~$C_{k-1},C_k$ are
  left unchanged and each column~$C_j$ for $j_{t-1} \le j \le k-2$ is
  replaced by $C_j \oplus C_{j+1}$.  The first row remains unchanged and
  all entries of the second row but the last two become $0$.  We apply $t$
  transformations like this using successively $j_t,j_{t-1},\ldots,j_1$.
  Then $k-t-1$ further steps are made using then $k_0 = 1$ each time.
  After applying all these transformations, each entry~$P_{i,j}$ for
  $i+j = \ell+1$ satisfies $P_{i,j} = 1$ and each entry~$P_{i,j}$ for
  $i+j < \ell+1$ satisfies $P_{i,j} = 0$.  It follows that the determinant
  of~$P$ is~$1$ and that the matrix~$P$ is invertible.
\end{proof}

For a word $w$ we write $w^n$ to denote the word given by concatenation of
$n$ copies of~$w$.  The following lemma states that each $(k,m)$-nested
perfect necklace can be transformed into another $(k,m)$-nested perfect
necklace which starts with $0^m$.

\begin{lemma} \label{lem:xor-npn}
  Let $w$ be a word of length $m2^k$ and let $z$ be a word of length~$m$.
  The word~$w$ is a $(k,m)$-nested perfect necklace if and only if the word
  $w \oplus z^{2^k}$ is a $(k,m)$-nested perfect necklace.
\end{lemma}
\begin{proof}
  Note first that both words $w$ and~$z^{2^k}$ have a length of~$m2^k$.
  Let $w'$ be the word $w \oplus z^{2^k}$.  Let $\ell$ be an integer such
  that $1 \le \ell \le k$ and let $v'$ be a block of~$w'$ of
  length~$m2^{\ell}$ starting at a position~$j$ congruent to~$1$ modulo
  $m2^{\ell}$.  The corresponding block of~$w$ at the same position~$j$ is
  of course $v = v' \oplus z^{2^\ell}$.  By hypothesis, this later
  block~$v$ is a $(\ell,m)$-nested perfect necklace.  We claim that $v'$ is
  also a $(\ell,m)$-nested perfect necklace.

  Let $i$ be such that $1 \le i \le m$ and let $u'$ any word of
  length~$\ell$.  Let $t$ be the block of~$zz$ of length~$\ell$ starting at
  position~$i$ and consider the word $u = u' \oplus t$.  This word~$u$ has
  an occurrence in the necklace~$v$ at a position~$j'$ congruent to~$i$
  modulo~$m$.  It follows that $u' = u \oplus t$ has an occurrence at the
  same position~$j'$ in~$v'$.  Since each word~$u$ has such an occurrence
  for each possible~$i$ and $v'$ has length $m2^{\ell}$, $v'$ is a
  $(\ell,m)$-nested perfect necklace.
\end{proof}

We can now prove that the  all the $(k,m)$-affine necklaces are  $(k,m)$-nested  
perfect necklaces.

\begin{proposition} \label{pro:aff2nes}
  Let $k,m,d$ be integers such that $d \ge 0$, $m = 2^d$ and
  $1 \le k \le m$.  Each $(k,m)$-affine necklace is a $(k,m)$-nested
  perfect necklace.
\end{proposition}
The proof of Proposition \ref{pro:aff2nes} follows and extends that
of~\cite[Lemma~5]{Levin1999} but we use a different notation.
\begin{proof}
  It suffices of course to prove the result for $k = m$.  By
  Lemma~\ref{lem:xor-npn}, it may be assumed that the vector~$z$ in the
  definition of affine necklaces is the zero vector.  Let $M$ be one of the
  matrices $M_d^{n_1,\ldots,n_m}$, let $w_1,\ldots,w_{2^m}$ be the
  enumeration in lexicographic order of all words of length~$m$ over~$\FT$
  and suppose that the $(m,m)$-affine necklace~$w$ is the concatenation
  $(Mw_1)(Mw_2) \cdots (Mw_{2^m})$.  Let $k$ be an integer such that
  $1 \le k \le m$ and let $w'$ be a block of~$w$ of length~$m2^k$ starting
  at a position congruent to~$1$ modulo $m2^k$.  The word~$w'$ is thus
  equal to a concatenation of the form
  $(Mw_{p2^k+1}) \cdots (Mw_{(p+1)2^k})$ for some fixed integer~$p$ such
  that $0 \le p \le 2^{m-k}-1$.  We claim that $w'$ is a $(k,m)$-perfect
  necklace.

  To prove the claim, it must be shown that for each
  integer~$\ell$ such that $0 \le \ell < m$, each word~$u$ of length~$k$
  has exactly one occurrence in~$w$ with a starting position congruent
  to~$\ell+1$ modulo~$m$ (we write $\ell+1$ rather than $\ell$ because
  positions are numbered from~$1$).  We now suppose that the word $u$ and
  the integer~$\ell$ such that $0 \le \ell < m$ are fixed.  We distinguish
  two cases depending on whether $\ell+k \le m$ or~not.

  We first suppose that $k + \ell \le m$.  It follows that the wanted
  occurrence of~$u$ must be fully contained in a single word~$Mw_{p2^k+q}$
  for $1 \le q \le 2^k$.  More precisely it must lie in the positions
  $\ell+1, \ldots, \ell+k$ of $Mw_{p2^k+q}$. In that case, the claim boils
  down to showing that there is exactly one integer~$q$ such that $u$
  occurs in positions $\ell+1, \ldots, \ell+k$ of $Mw_{p2^k+q}$.  Let us
  recall that $p$ is fixed and that $q$ ranges in $1,\ldots, 2^k$.  Since
  $w_i$ is the base~$2$ expansion of~$i-1$ with $m$ bits, $w_{p2^k+q}$ can
  factorized $x_py_{q-1}$ where $x_p$ and $y_{q-1}$ are the base~$2$
  expansions of $p$ and~$q-1$ with $m-k$ and $k$ bits.
  
  \begin{center}
    \begin{tikzpicture}
      \node at (-0.65,0) {$M = $};
      \node at (0,0) {$\left(\vphantom{\rule{0cm}{1cm}} \right.$};
      \node at (2,0) {$\left.\vphantom{\rule{0cm}{1cm}} \right)$};
      \node at (1.55,0.1) {$P$};
      \draw (1.15,-0.3) rectangle (1.95,0.5);
      \node at (0.6,0.1) {$N$};
      \draw (0.05,-0.3) rectangle (1.1,0.5);
    \end{tikzpicture}
  \end{center}

  The occurrence of~$u$ in~$w_{p2^k+q}$ is now translated into linear
  equations by introducing the following two matrices $N$ and~$P$ (see
  above).  Let $N$ and $P$ be the following sub-matrices of the matrix~$M$.
  The $k \times m-k$ matrix~$N$ is obtained by selecting the $k$ rows
  $L_{\ell+1},\ldots, L_{\ell+k}$ and the $m-k$ columns
  $C_1,\ldots,C_{m-k}$.  The $k \times k$ matrix~$P$ is obtained by
  selecting the same $k$ rows $L_{\ell+1},\ldots, L_{\ell+k}$ and the $k$
  columns $C_{m-k+1},\ldots,C_{m}$.  The word~$u$ occurs in the positions
  $\ell+1, \ldots, \ell+k$ of $Mw_{p2^k+q}$ if and only if
  $u = Nx_p + Py_{q-1}$ where words $u$, $x_p$ and~$y_{q-1}$ are considered
  as columns vectors of respective dimensions $k$, $m-k$ and~$k$.  Since
  $x_p$ is fixed and $P$ is invertible by Lemma~\ref{lem:invert-right-sm},
  there is exactly one solution for~$y_{q-1}$ and thus one solution
  for~$q$.  This proves the claim when $k + \ell \le m$.

  We now suppose that $\ell+k>m$.  The wanted occurrence of~$u$ must then
  overlap two consecutive words $Mw_{p2^k+q}$ and $Mw_{p2^k+q+1}$ where
  $p2^k+q+1$ should be understood as $p2^k+1$ if $q = 2^k$.  Let us write
  $u = u_1u_2$ where $u_1$ and $u_2$ have length $m-\ell$ and $\ell+k-m$.
  The wanted occurrences exist if $u_1$ occurs at positions
  $\ell+1,\ldots,m$ of $Mw_{p2^k+q}$ and $u_2$ occurs at positions
  $1,\ldots, \ell+k-m$ of $Mw_{p2^k+q+1}$ with the same convention for
  $p2^k+q+1$.  As in the previous case, these occurrences are translated
  into linear equations.  For that purpose, we introduce the following four
  matrices.

  \begin{center}
    \begin{tikzpicture}
      \node at (-0.65,0) {$M = $};
      \node at (0,0) {$\left(\vphantom{\rule{0cm}{1cm}} \right.$};
      \node at (2,0) {$\left.\vphantom{\rule{0cm}{1cm}} \right)$};
      \node at (1.55,0.40) {$P_2$};
      \draw (1.15,0.2) rectangle (1.95,0.67);
      \node at (0.6,0.40) {$N_2$};
      \draw (0.05,0.2) rectangle (1.1,0.67);
      \node at (1.55,-0.52) {$P_1$};
      \draw (1.15,-0.3) rectangle (1.95,-0.7);
      \node at (0.6,-0.52) {$N_1$};
      \draw (0.05,-0.3) rectangle (1.1,-0.7);
    \end{tikzpicture}
  \end{center}

  The matrices $N_1$ and $P_1$ are obtained by selecting the rows
  $L_{\ell+1},\ldots,L_m$ and the columns $C_1,\ldots,C_{m-k}$ for $N_1$
  and $C_{m-k+1},\ldots,C_m$ for $P_1$.  The matrices $N_2$ and $P_2$ are
  obtained by selecting the rows $L_1,\ldots,L_{\ell+k-m}$ and the columns
  $C_1,\ldots,C_{m-k}$ for $N_2$ and $C_{m-k+1},\ldots,C_m$ for $P_2$ (see
  above).  The two words $w_{p2^k+q}$ and $w_{p2^k+q+1}$ are then
  factorized $w_{p2^k+q} = x_py_{q-1}$ and $w_{p2^k+q+1} = x_py_q$ where
  $x_p$ is the base~$2$ expansion of~$p$ with $2^{m-k}$ bits and words
  $y_{q-1}$ and $y_q$ are the base~$2$ expansions of $q-1$ and~$q$
  (understood as~$0$ if $q = 2^k$) with $2^k$ bits.  The occurrences of
  $u_1$ and $u_2$ do exist as wanted if and only if these two equalities hold,
  
\begin{align*}
   u_1 & =  N_1x_p + P_1y_{q-1},  
  \\
   u_2 &= N_2x_p + P_2y_q  
  \end{align*}

  Notice that the first equation
  involves $y_{q-1}$ while the second one involves~$y_q$.  These two words
  are strongly related in the sense that each one determines the~other.
   For each $i$ such that  $0 \le i \le k$, the $i$
  right most bits of either $y_{q-1}$ or $y_q$ determine the $i$ right most
  bits of the other.  This is due to the fact that either adding or
  subtracting~$1$ can be performed on the bits from right to left.  For
  that reason, we will show that the equations $u_1 = N_1x_p + P_1y_{q-1}$
  and $u_2 = N_2x_p + P_2y_q$ have a unique solution in~$q$ by successively
  computing the bits of $q-1$ and~$q$ from right to~left.

  We actually describe a strategy for solving the two equations.  This
  strategy is based of the upper and lower borders of the matrix~$M$.  The
  main ingredient is that between two consecutive columns $C_j$
  and~$C_{j+1}$, one of the two borders uses a step which is not
  horizontal, that is, either North-East for the upper border or South-East
  for the lower border.

  The right most bit of $y_{q-1}$ and~$y_q$ can be found as follows.
  Either the upper border or the lower border makes a non horizontal step
  from~$C_{m-1}$ to~$C_m$.  It means that either the first row or the last
  row of~$M$ has the form $(0,\ldots,0,1)$.  This row can be used to find
  the right most bit of $y_{q-1}$ and~$y_q$ as it is the first row the
  equation $u_1 = N_1x_p + P_1y_{q-1}$ or the last row of the equation
  $u_2 = N_2x_p + P_2y_q$.  The second right most bit of $y_{q-1}$
  and~$y_q$ can be found as follows.  Either the upper border or the lower
  border makes a non horizontal step from~$C_{m-2}$ to~$C_{m-1}$.  It means
  that one row of~$M$ has the form $(0,\ldots,0,1,*)$.  It can be used to
  find the second right most bit of $y_{q-1}$ and~$y_q$ as it is one row
  of one of the two equations. 

 This process can be continued using at each
  step a row of either the first or the second equation.  In the process
  rows of the first equation are used from the first to the last while
  rows of the second equation are used from the last to the first.  This
  process can be continued until the rows of one the equations have been
  exhausted.  By symmetry, it can be assumed that all rows of the second
  equation have been used.  Suppose that the left most $n$ bits of $y_{q-1}$
  and~$y_q$ have still to be found.  Then the last $n$ rows of the first
  equations have not been used.  Considering the known bits as constants,
  the matrix involving the $r$ unknown bits is a matrix as in
  Lemma~\ref{lem:invert-top-sm}.  By this lemma, this matrix is invertible
  and these last $r$ bits can be found in a unique way.  This proves the
  claim when $\ell + k > m$ and finishes the proof of the proposition.
\end{proof}

\section{Nested perfect necklaces are affine necklaces}

We shall now show that all $(k,m)$-nested perfect necklaces
are $(k,m)$-affine necklaces.  Since the other inclusion has been already
proved, it suffices to show that they have the same cardinality.  
 The next lemma shows that for $k = 1$, they coincide.

\begin{lemma} \label{lem:1m-necklaces}
  The $(1,m)$-nested perfect necklaces are the words of the form $ww'$
  where $w$ and~$w'$ are two words of length~$m$ satisfying
  $w' = w \oplus 1^m$.  Furthermore, they are all affine.
\end{lemma}
\begin{proof}
  It is straightforward that $(1,m)$-nested perfect necklaces are the
  words  of the  form stated in the lemma.  And for each matrix $M$ of the form
  $M_d^{n_1,\ldots,n_m}$, $Mw_0$ and $Mw_1$ are respectively equal to
  $0^m$ and $1^m$.  This proves the last claim.
\end{proof}

The next lemma provides the number of $(m,m)$-affine necklaces.  It shows
that $(m,m)$-affine necklaces obtained by the different choices of the matrix
$M_d^{n_1,\ldots,n_m}$ and of the vector~$z$ are indeed different.
\begin{lemma} \label{lem:mm-necklaces}
  Let $m = 2^d$ for some $d \ge 0$.  There are exactly $2^{2m-1}$
  different $(m,m)$-affine necklaces.
\end{lemma}
\begin{proof}
  There are exactly $2^{m-1}$ matrices $M_d^{n_1,\ldots,n_m}$.  Indeed, the
  sequence $n_1,\ldots,n_m$ is fully determined by the sequence
  $n_1-n_2,\ldots,n_{m-1}-n_m$ of $m-1$ differences which take their value
  in~$\{0,1\}$.  There are also $2^m$ possible values for the word~$z$
  in~$\FT^m$.  This proves that the number of $(m,m)$-affine necklaces
  is bounded by~$2^{2m-1}$.

  It remains to show that two $(m,m)$-affine necklaces obtained for two
  different pairs $(M,z)$ and $(M',z')$ are indeed different.  Let
  $w_1,\ldots,w_{2^d}$ be the enumeration in lexicographic order of all
  words of length~$m$ over~$\FT$. Let $M$ and $M'$ be two matrices of the
  form $M_d^{n_1,\ldots,n_m}$.  Let $z$ and~$z'$ be two words over~$\FT$ of
  length~$m$ and let $u_i = w_i \oplus z$ and $u'_i = w_i \oplus z'$ for
  $1 \le i \le 2^m$.  Let $w$ and $w'$ be the two concatenations
  $(Mu_1) \cdots (Mu_{2^d})$ and $(M'u'_1) \cdots (M'u'_{2^d})$.
  We claim that if $w = w'$, then $M = M'$ and~$z = z'$.
  
  We suppose that $w = w'$.  Since both matrices $M$ and~$M'$ are
  invertible by Lemma~\ref{lem:invert-right-sm}, $Mu_i$ (respectively
  $M'u'_i$) is the zero vector if and only if $u_i$ (respectively $u'_i$)
  is the zero vector, that is, $z = w_i$ (respectively $z' = w_i$).  It
  follows then that $z = z'$ and thus $u_i = u'_i$ for $1 \le i \le 2^m$.
  Note that the vector $u_i$ ranges over all possible vectors of
  length~$m$.  If $Mu_i = M'u_i$ for all $1 \le i \le 2^m$, then~$M = M'$.
\end{proof}

Lemma~\ref{lem:extend2} will  show how $(k,m)$-affine necklaces
can be concatenated with   $(k,m)$-affine necklaces to get  $(k+1,m)$-perfect
necklaces.  The next two lemmas are intermediate steps towards the proof.
The first states that each rotation of a column of~$M_d$ is a linear
combination of some columns to its right.
\begin{lemma} \label{lem:right-comb}
  Let $d \ge 0$ be integer and let $(C_1,\ldots,C_{2^d})$ be the columns of
  the matrix~$M_d$.  For any integers $i,k$ such that $1 \le i \le 2^d$
  and $k \ge 0$, the vector $\sigma^k(C_i) \oplus C_i$ is equal to a linear
  combination $\bigoplus_{j = i+1}^{2^d} b_jC_j$ where $b_j \in \FT$.
\end{lemma}
\begin{proof}
  The result is proved by the induction on the difference $2^d - i$.  If
  $i = 2^d$, the result holds trivially because
  $\sigma(C_{2^d}) = C_{2^d}$.  Assume that $i < 2^d$ is fixed.  The
  proof is now by induction on the integer~$k$.  The result for $k = 0$ is
  void.  By Lemma~\ref{lem:pascal} applied to the column $C_{i+2^d}$ of the
  matrix $M_{d+1}$, the equality $\sigma(C_i) \oplus C_i = C_{i+1}$ holds.
  We apply Lemma~\ref{lem:pascal} to the column $C_{i+2^d}$ of the
  matrix~$M_{d+1}$ because this column has period $2^d$ and its first half
  is the column~$C_i$ of~$M_d$.  This proves the result for $k = 1$.
  Suppose now that the result is true for some~$k \ge 1$.  Applying
  $\sigma$ to both terms of the equality and replacing first $\sigma(C_i)$
  by the value $C_i \oplus C_{i+1}$ and second each $\sigma(C_j)$ by the
  value given by the induction hypothesis gives the result for~$k+1$.
\end{proof}

The next lemma shows for each $(k,m)$-affine necklace, there is just one
possible way of rotating it to get another $(k,m)$-affine necklace.
\begin{lemma} \label{lem:affine-conj}
  Let $d,m,k$ and $p$ be integers such that $d \ge 0$, $m = 2^d$,
  $1 \le k \le m$ and $p \ge 0$.  Let $w$ be a $(k,m)$-affine
  necklace.  If $m$ divides $p$ and $\sigma^p(w)$ is also a $(k,m)$-affine
  necklace, then $p \equiv m2^{k-1} \mod |w|$.
\end{lemma}

\begin{proof}
  Since $|w| = m2^k$ and $\sigma^{|w|}(w) = w$, we may assume that $0 \le
  p \le  m2^k$.  The result holds if either $p = 0$ or $p = m2^k$.
  Therefore we assume that $1 \le p \le m2^k-1$.
  Let $w_1,\ldots, w_{2^d}$ be the enumeration in lexicographic order of
  all words over~$\FT$ of length~$m$.
  Since $w$ is an affine necklace, it is a concatenation
  $(Mu_1) \cdots M(u_{2^k})$ where $M$ is a matrix $M_d^{n_1,\ldots,n_m}$
  and $u_i$ is equal to $w_i \oplus z$ for each integer $1 \le i \le 2^k$
  and for some fixed vector~$z$.
  Since $\sigma^p(w)$ is also an affine necklace, it is a concatenation
  $(M'u'_1) \cdots (M'u'_{2^k})$ where $M'$ is a matrix
  $M_d^{n_1,\ldots,n_m}$ and $u'_i$ is equal to $w_i \oplus z'$ for each
  integer $1 \le i \le 2^k$ and for some other fixed vector~$z'$.  For each
  $\ell$ such that $0 \le \ell < d$, the vector
  $M'u'_1 \oplus M'u'_{1+2^\ell}$ is equal to the column $C'_{m-\ell}$ of the
  matrix~$M'$.  Since $M$ and $M'$ are two matrices of the form
  $M_d^{n_1,\ldots,n_m}$, the column $C'_{m-\ell}$ of~$M'$ is equal to
  $\sigma^t(C_{m-\ell})$ where $C_{m-\ell}$ is the corresponding column
  of~$M$ and $t$ is some integer.
 Since $m$ divides $p$, the necklace~$\sigma^p(w)$ is equal to
  \[
   (Mw'_i) \cdots M(w'_{2^k})(Mw'_1) \cdots (Mw'_{i-1})
  \] where
  $i = 1 + p/m$. We consider the word~$w_i$ which is the base~$2$ expansion
  of~$i-1$ with $m$ bits.  Let $\ell$ be the greatest integer such that
  $2^{\ell}$ divides $i-1 = p/m$.  The integer~$i-1$ is equal to
  $2^{\ell}(2r+1)$ for some non-negative integer~$r$.  
   We claim that~$\ell = k-1$.

  Suppose by contradiction that $\ell < k-1$.  The integer~$r$ satisfies
  thus $r \ge 1$.  The word~$w_i$ is then equal to $0^{m-k}u10^{\ell}$
  where $0^{m-k}$ is the block leading zeros due to $i \le 2^k$ and $u$ is
  the base~$2$ expansion of~$r$ with $k-\ell-1$ digits.  We now consider
  the word $w_{i+2^{\ell}}$.  This word is equal to $0^{m-k}u'0^{\ell+1}$
  where $u'$ is the base~$2$ expansion of~$r+1$ with $k-\ell-1$ digits.
  Computing $Mu_i \oplus Mu_{i+2^{\ell}}$ gives $C_{2^d-\ell} + R$ where
  $R$ is a non-zero linear combination of $C_{m-k},\ldots, C_{m-\ell-1}$.
  This linear combination~$R$ cannot be equal to zero because the words $u$
  and~$u'$ are different.  The vector $Mu_i \oplus Mu_{i+2^{\ell}}$ is also
  equal to
  $M'u'_1 \oplus M'u'_{1+2^{\ell}} = C'_{m-\ell} = \sigma^t(C_{m-\ell})$.
  By Lemma~\ref{lem:right-comb}, this vector is equal to $C_{m-\ell} + R'$
  where $R'$ is a linear combination of $C_{m-\ell+1},\ldots,C_m$.  This is
  a contradiction: the equality $R = R'$ is impossible because,  
  by Lemma~\ref{lem:invert-right-sm},
  the matrix~$M$ is invertible.
\end{proof}

We are now ready to show that each $(k,m)$-affine necklace can be extended
by at most two $(k,m)$-affine necklaces to get a $(k+1,m)$-perfect necklace.

\begin{lemma} \label{lem:extend2}
  Let $m = 2^d$ for some $d \ge 0$ and let $k$ be an integer such that
  $1 \le k \le m$.  Let $w$ be a $(k,m)$-affine necklace.  There are at
  most two $(k,m)$-affine necklaces~$w'$ such that $ww'$ is a
  $(k+1,m)$-perfect necklace.
\end{lemma}
\begin{proof}
  We use the characterization of $(k,m)$-perfect necklaces as cycles in
  appropriate graphs $G_k$ (which variant of de Bruijn graphs) given
  in~\cite{AlBeFeYu2016}.  Consider the directed graph~$G_k$ whose vertex
  set is $\FT^k \times \{ 1, \ldots, m\}$ and whose transitions are defined
  as follows.  There is an edge in~$G_k$ from $(u,i)$ to $(u',i')$ if first
  there two symbols $a$ and~$b$ in~$\FT$ such that $ua = bu'$ and second
  $i' \equiv i+1 \mod m$.  The condition on $u$ and~$u'$ means that $u$
  and~$u'$ are respectively the prefix and the suffix of length~$k$ of the
  word $v = ua = bu'$ of length~$k+1$.  Therefore, the edges of the
  graph~$G_k$ can be identified with the words of length~$k+1$ over~$\FT$.
  Note that each vertex of~$G_k$ has two incoming and two outgoing edges.
    It follows from the definition of the graph~$G_k$, that each
  $(k,m)$-nested perfect necklace is 
   identified to a Hamiltonian cycle  in~$G_k$ and that each $(k+1,m)$-nested 
  perfect necklaces is 
  identified to an  Eulerian cycle in~$G_k$.  

  Let $w$ be a $(k,m)$-affine necklace. Then $w$ determines a Hamiltonian
  cycle~$C$ in~$G_k$.  Since $C$ visits each node of~$G_k$ exactly once, it
  uses one outgoing edge of each node.  Any $(k,m)$-nested perfect
  necklace~$w'$ such that $ww'$ is a $(k+1,m)$-nested perfect necklace
  induces an Hamiltonian~$C'$ cycle which cannot use an edge of~$C$.
  Otherwise, it would not be possible to build an Eulerian cycle from $C$
  and~$C'$ and $ww'$ would not be a $(k+1,m)$-nested perfect necklace.  If
  there is no $(k,m)$-affine necklace~$w'$ such that $ww'$ is a
  $(k+1,m)$-nested perfect necklace the lemma trivially holds.  Suppose now
  that there exists at least one such~$w'$.  Since the graph
  $G_k \setminus C$ has only one outgoing edge from any vertex, any
  $(k,m)$-nested perfect necklace $w''$ such that such that $ww''$ is a
  $(k+1,m)$-nested perfect necklace must be of the form $\sigma^p(w')$ for
  some integer~$p \ge 0$.  Since both Hamiltonian cycles $C'$ and~$C''$
  induced by $w'$ and~$w''$ must must start from a vertex in
  $\FT^k \times \{1\}$, it follows that $m$ divides $m$.  By
  Lemma~\ref{lem:affine-conj}, the only possible choices for $p$ are $0$
  and~$m2^{k-1}$.  This proves that there is at most one such~$w''$
  different from~$w'$.
\end{proof}

We can now give the  number of $(k,m)$-affine necklaces.
\begin{proposition} \label{prop:nbr-nested}
  Let $m = 2^d$ for some $d \ge 0$.  For each integer~$k$ such that
  $1 \le k \le m$, the number of $(k,m)$-affine necklaces is exactly
  $2^{k+m-1}$.
\end{proposition}
\begin{proof}
  We assume the integer~$m$ to be fixed and we let $t_k$ denote the number
  of $(k,m)$-affine necklaces.  By Lemma~\ref{lem:1m-necklaces}, $t_1$ is
  equal to~$2^m$ and by Lemma~\ref{lem:mm-necklaces}, $t_m$ is equal
  to~$2^{2m-1}$.  It follows from Lemma~\ref{lem:extend2} that
  $t_{k+1} \le 2 t_k$ for each integer~$k$ such that $1 \le k < m$.  None of
  these inequalities can be strict because otherwise $t_m$ would be
  striclty less that $2^{2m-1}$.
  So,  for each integer  $k$ such that 
  $1 \le k \le m$,  $t_{k+1} = 2t_k$,   hence, $t_k = 2^{k+m-1}$.
\end{proof}

The results above allows us to prove the wanted inclusion.
\begin{proposition} \label{pro:nes2aff}
  Each $(k,m)$-nested perfect necklace is a $(k,m)$-affine necklace.
\end{proposition}
\begin{proof}
  Fix $m$, let  $s_k$ be  the number  of $(k,m)$-nested perfect necklaces 
 and let $t_k$ be the number of  $(k,m)$-affine necklaces.  
  By Proposition~\ref{pro:aff2nes},
  $s_k \le t_k$ holds for each integer~$k$ such that $1 \le k \le m$.  To
  prove the statement, it suffices to prove that $s_k = t_k$ for each
  integer~$k$ such that $1 \le k \le m$.  
  We prove it by induction on $k$. By Lemma~\ref{lem:1m-necklaces},
  $s_1 = t_1 = 2^m$.  We now suppose $s_k = t_k$ and we prove that
  $s_{k+1} = t_{k+1}$.  Each $(k+1,m)$-nested perfect necklace can be
  written as $ww'$ where $w$ and~$w'$ are two $(k,m)$-nested perfect
  necklaces.  Since $s_k = t_k$, $w$ and~$w'$ are also $(k,m)$-affine
  necklaces.  By Lemma~\ref{lem:extend2}, there are at most two possible
  choices of~$w'$ for each~$w$.  This proves that $t_{k+1} \le 2t_k$.
  Since $s_{k+1} = 2s_k$ by Proposition~\ref{prop:nbr-nested} and
  $s_{k+1} \le t_{k+1}$, the equality $s_{k+1} = t_{k+1}$ holds.
\end{proof}

\section*{Acknowledgements}

The authors are members of the Laboratoire International Associ\'e INFINIS,
CONICET/Universidad de Buenos Aires--CNRS/Universit\'e Paris Diderot
and they are partially supported by the ECOS project PA17C04.
Carton is partially funded by the DeLTA project (ANR-16-CE40-0007).

\bibliographystyle{plain}
\bibliography{necklaces}

\begin{thebibliography}{10}

\bibitem{AlBeFeYu2016}
N.~{\'A}lvarez, V.~Becher, P.~A. Ferrari, and S.~A. Yuhjtman.
\newblock Perfect necklaces.
\newblock {\em Advances of Applied Mathematics}, 80:48--61, 2016.

\bibitem{Bugeaud2012}
Y.~Bugeaud.
\newblock {\em Distribution modulo one and {D}iophantine approximation}, volume
  193 of {\em Cambridge Tracts in Mathematics}.
\newblock Cambridge University Press, Cambridge, 2012.

\bibitem{deBruijn1946}
N.~G. de~Bruijn.
\newblock A combinatorial problem.
\newblock {\em Koninklijke Nederlandse Akademie v.Wetenschappen}, 49:758--764,
  1946.
\newblock Indagationes Mathematicae 8 (1946) 461-467.

\bibitem{DrmotaTichy1997}
M.~Drmota and R.~Tichy.
\newblock {\em Sequences, Discrepancies and Applications}.
\newblock Lecture Notes in Mathematics, Vol. 1651. Springer-Verlag, 1997.

\bibitem{Faure1982}
H.~Faure.
\newblock Discr{\'e}pance de suites associ{\'e}es {\`a} un syst{\`e}me de
  num{\'e}ration (en dimension $s$).
\newblock {\em Acta Arithmetica}, 41, 1982.

\bibitem{Korobov1955}
N.M. Korobov.
\newblock Numbers with bounded quotient and their applications to questions of
  diophantine approximation.
\newblock {\em Izv. Akad. Nauk SSSR, Ser. Mat.}, 19:361--380, 1955.

\bibitem{KuipersNiederreiter2006}
L.~Kuipers and H.~Niederreiter.
\newblock {\em Uniform distribution of sequences}.
\newblock Dover Publications, Inc., New York, 2006.

\bibitem{Levin1995}
M.~B. Levin.
\newblock On the upper bounds of discrepancy of completely uniform distributed
  and normal sequences.
\newblock {\em Abstracts American Mathematical Society}, 16:556–557, 1995.
\newblock AMS-IMU joint meeting, Jerusalem, Israel, May 24–26, 1995.

\bibitem{Levin1999}
M.~B. Levin.
\newblock On the discrepancy estimate of normal numbers.
\newblock {\em Acta Arithmetica}, 88(2):99--111, 1999.

\bibitem{Schmidt1972}
W.~Schmidt.
\newblock Irregularities of distribution. vii.
\newblock {\em Acta Arithmetica}, 21:45–50, 1972.

\end{thebibliography}

\bigskip
\bigskip

{\small
\begin{minipage}{\textwidth}
\noindent
Ver\'onica Becher \\
Departamento de  Computaci\'on,
Facultad de Ciencias Exactas y Naturales \& ICC \\
Universidad de Buenos Aires \&  CONICET, Argentina \\
vbecher@dc.uba.ar
\bigskip\\
Olivier Carton \\
Institut de Recherche en Informatique Fondamentale \\
Universit\'e Paris Diderot, France \\
Olivier.Carton@irif.fr
\end{minipage}
}

\end{document}